\documentclass[10pt,oneside,reqno]{amsart}
\usepackage{amsmath,amssymb}
\usepackage{amsthm,comment}
\usepackage{mathtools}
\usepackage{enumerate}
\usepackage{hyperref}

\usepackage{amssymb}
\usepackage{mathtools}

\newtheorem{thm}{Theorem}[section]
\newtheorem{cor}[thm]{Corollary}
\newtheorem{lemma}[thm]{Lemma}
\newtheorem{prop}[thm]{Proposition}
\theoremstyle{definition}
\newtheorem{defn}[thm]{Definition}
\theoremstyle{remark}
\newtheorem{rem}[thm]{Remark}
\numberwithin{equation}{section}
\newtheorem{exa}[thm]{Example}
\newtheorem{conj}[thm]{Conjecture}
\newtheorem{todo}{ToDo}

\newtheorem{note}{Note}

\def\<{\prec}
\def\>{\succ}
\def\O{{\mathcal O}}
\def\bes{\begin{equation*}}
\def\be{\begin{equation}}
\def\ee{\end{equation}}
\def\ees{\end{equation*}}
\def\dif{\operatorname{d}\!}

\allowdisplaybreaks

\def\bN{\mathbb{N}}
\def\bR{\mathbb{R}}

\def\cR{\mathcal{R}}

\def\arcsinh{\operatorname{arcsinh}}

\newcommand\abs[1]{\lvert{#1}\rvert}

\def\ba{{\bf a}}

\def\bg{{\bf g}}
\def\bh{{\bf h}}
\def\bd{{\bf d}}

\def\bs{{\bf s}}

\begin{document}

\title[Comparison and sub/super-stabilizability of some new means]{Results on comparison and sub/super-stabilizability of some new means}
\author{Lenka Mihokovi\'c and Mustapha Ra\"{i}ssouli}


\address{Lenka Mihokovi\'c, University of Zagreb, Faculty of Electrical Engineering and Computing, Unska 3, 10000 Zagreb, Croatia}
\email{lenka.mihokovic@fer.hr}

\address{Mustapha Ra\"{i}ssouli, Department of Mathematics, Faculty of Science, Moulay Ismail University, Meknes, Morocco.}
\email{raissouli.mustapha@gmail.com}


\maketitle

\begin{abstract}
	 We present analysis of some new means recently introduced by M.\ Ra\"{i}ssouli and A.\ Rezgui.
	 We establish comparison relations and results on $(K,N)$-sub/super-stabilizability where $K$ and $N$ belong to the class of power means, denoted by $B_p$, 
and $M$ is one of the classical or recently studied new means.
	 Assuming that means $K$, $M$ and $N$ have asymptotic expansions, we present the complete asymptotic expansion of the resultant mean-map.	
	 As an application of the obtained asymptotic expansions and the asymptotic inequality between $M$ and $\cR(B_p,M,B_q)$, we show how to find the optimal parameters $p$ and $q$ for which $M$ is $(B_p,B_q)$-sub/super-stabilizable.

\end{abstract}

\section{Introduction}


Through this paper we consider \emph{bivariate mean}, i.e.\ a function
$M\colon \bR^+ \times \bR^+\to\bR^+$
such that
$
	 \min(s,t) \le M(s,t) \le \max(s,t).
$
We say that mean $M$ is \emph{symmetric} if
$ M(s,t) =M(t,s)$ for all $s,t>0$,
and \emph{homogeneous} (of degree 1) if
$ M(\lambda s,\lambda t)=\lambda M(s,t)$ for all $\lambda,s,t,>0.$
For three homogeneous symmetric bivariate means $K$, $M$ and $N$, we define the so-called  resultant mean-map
\bes
	\cR(K,M,N)(s,t)=K\Bigl( M\bigl(s,N(s,t)\bigr),M\bigl(N(s,t),t\bigr) \Bigr).
\ees
A symmetric mean $M$ is said to be \emph{stable} (\emph{balanced}), if $\cR(M,M,M)=M$.

 \begin{defn}[\cite{Raiss-2011-stab,RaissSandor-2014}]\label{defsubsuperst}
 	Let $K$, $N$ be two nontrivial stable means.
	Mean $M$ is called
	\begin{enumerate}
		\item \emph{$(K,N)$-stabilizable}, if the following relation is satisfied:
			\bes\label{def-stabilizable}
			M(s,t) =\cR(K,M,N)(s,t) = K\Bigl( M\bigl(s,N(s,t)\bigr),M\bigl(N(s,t),t\bigr) \Bigr).
			\ees
 		\item \emph{$(K,N)$-sub-stabilizable}, if $\cR(K,M,N)\le M$ and $M$ is between $K$ and $N$,
 		\item \emph{$(K,N)$-super-stabilizable}, if $M\le\cR(K,M,N)$ and $M$ is between $K$ and $N$.
 	\end{enumerate}
 \end{defn}

Some interesting results regarding these notions can also be found in papers \cite{AnAn-2014,RaissSandor-2016}.
Motivated by the results on sub/super-stabilizability for bivariate means studied in the paper \cite{RaissSandor-2014} in combination with the general algorithms for the coefficients in the asymptotic expansion of the resultant mean-map and consequently of the stabilizable mean, obtained in the paper \cite{Mih-2023-sssa}, we present the results on sub/super-stabilizability in the context of the asymptotic expansions.

\begin{defn}
		For an asymptotic sequence of functions $(\varphi_n)_{n\in\bN_0}$ the (formal) series $\sum_{n=0}^{\infty} a_n \varphi_n(x)$ is said to be an \emph{asymptotic expansion} of a function $f(x)$ as $x\to x_0$ if for each $N\in\bN_0$
		$$
			f(x)=\sum_{n=0}^{N}a_n \varphi_n(x) +o(\varphi_N(x)).
		$$
\end{defn}
Theoretical background from theory of asymptotic expansions can be found in \cite{Erd}.
Through this paper, we use asymptotic expansion with respect to the asymptotic sequence $\varphi_n(x) = x^{1-n}$, $n\in\bN_0$, as $x\to\infty$.
Based on the sign of the first term in such asymptotic expansion we introduce the notion of the asymptotic inequality.
	\begin{defn}[\cite{Vu}]
	Let $F(s,t)$ be any homogeneous bivariate function such that
	\bes
		F(x+s,x+t)
		=c_{k}(t,s)x^{-k+1}+\O(x^{-k}).
	\ees
	If $c_{k}(s,t)>0$ for all $s$ and $t$, then
	we say $F$ is \textit{asymptotically} greater than zero, and write
	\bes
		F\succ 0.
	\ees
\end{defn}
Asymptotic inequality is the necessary condition for the proper inequality,
i.e.\ if $F\ge0$, then $F\succ 0$.

The subject of study in this paper are means from \cite{RaissRezgui-2019} alone for themselves and also in combinations with power means. Recall the definition of the $r$-th power mean
	\bes
		B_r(a,b)=
			\begin{cases}\displaystyle
			\biggl(\frac{a^r+b^r}{2}\biggr)^{1/r},\qquad &r\ne0,\\
			\sqrt{ab},&r=0.
		\end{cases}
	\ees
This class of means covers some well known classical means such as arithmetic mean $A=B_1$, geometric mean $G=B_0$ and harmonic mean $H=B_{-1}$. It has been proved (\cite{Raiss-2011-stab}) that power means $B_p$ are stable for all real values of parameter $p$. Regarding its asymptotic expansion which has been studied in \cite{ElVu-2014-04}, one may find that
\be\label{asymexp-powermean}
	B_p(x-t,x+t)\sim x + \tfrac12(p-1) t^2 x^{-1} +\tfrac1{24}(p-1)(3+p-2p^2)t^4 x^{-3}+\O(x^{-5}).
\ee
\begin{rem}\label{rem-tt-st}
	This so-called one variable asymptotic expansion, i.e.\ asymptotic expansion of a mean in variables $(x-t,x+t)$, is sufficient to determine completely the two variable asymptotic expansion, i.e.\ asymptotic expansion in variables $(x+s,x+t)$ as it was proved in \cite[Lemma 2.1.]{BuMih-2023}. Following the same procedure as in the mentioned Lemma, we may conclude that the first non-zero coefficient in both of those expansions is the same.
\end{rem}

For the convenience of the reader, let us list all the means from \cite{RaissRezgui-2019} which will be involved in analysis in this paper. Means from the mentioned paper can be written in a form
\bes
	m_f(a,b)=\frac{2(a-b)}{f(\frac{a}{b})-f(\frac{b}{a})},\ a\neq b,\qquad m_f(a,a)=a.
\ees
We are interested in some of the special cases.
Let
\bes
	g(x)=\begin{cases}
		\mu(\ln x), &x\ge1,\\
		-\mu(\ln\frac1x), &0<x<1,
		\end{cases}
	\quad
	\mu(x)=\int_0^xu(t) \dif   t \text{ odd function},
\ees
and
\be\label{mean-type-mg}
	m_g(a,b)=\frac{\lvert a-b\rvert}{\mu(\lvert\ln\frac{a}{b}\rvert)} = M_u(a,b).
\ee
With $u(x)=\cosh(\alpha x)$, and for $\lvert \alpha \rvert\le1$, we have
\be\label{mean-def-Lalpha}
	\quad L_\alpha(a,b) \coloneqq M_u(a,b)=\frac{2\alpha a^\alpha b^\alpha (b-a)}{b^{2\alpha}-a^{2\alpha}}.
\ee
When $u(x)=\frac1{\cosh(\alpha x)}$, and for $\lvert \alpha \rvert\le1$, we have
\be\label{mean-def-Salpha}
		S_\alpha(a,b)\coloneqq M_u(a,b)
		=\frac{2\alpha(b-a)}{4\arctan(\frac{b}{a})^\alpha-\pi}
		=\frac{\alpha(b-a)}{2\arctan\frac{b^{\alpha}-a^{\alpha}}{b^{\alpha}+a^{\alpha}}}.
\ee
Remark that $L_{-\alpha}=L_\alpha$ and $S_{-\alpha}=S_\alpha$. We can then assume that $0\leq\alpha\leq1$. The means $L_{\alpha}$ and $S_{\alpha}$ cover some well-known standard means and their properties are embodied in the following result. The first part can be seen by a simple verification and for the second part see \cite{RaissRezgui-2019}.

\begin{prop}\label{pLS}
The following statements are valid:
\begin{enumerate}[(i)]
\item \label{prop-pLS-a} $L_0=S_0=L$, $L_{1/2}=G$, $L_1=H$, $S_{1/2}=P$, $S_1=T$ and $L_{1/4}=HZ_{1/4}$, where $L$, $G$, $H$ denote logarithmic, geometric and harmonic mean,
$P$ and $T$ refer to the first and the second Seiffert mean respectively, and $HZ_{1/4}$ is the Heinz mean.
\item \label{prop-pLS-b} For fixed $a\neq b$, the map $\alpha\longmapsto L_{\alpha}(a,b)$ is strictly decreasing in $\alpha\in[0,1]$ while $\alpha\longmapsto S_{\alpha}(a,b)$ is strictly increasing in $\alpha\in[0,1]$.
\end{enumerate}
\end{prop}

We also recall the definitions of the following means which were introduced in \cite[Examples 4.4--4.7]{RaissRezgui-2019}, where $a,b>0, a\neq b$, 
$r>0$, $\lvert \alpha\rvert\le1$:
\be\label{mean-def-Mi}
\begin{aligned}
	M_1(a,b)& \coloneqq \frac{\abs{b-a}}{\ln(1+\abs{\ln b-\ln a})},\\
	M_2(a,b)&\coloneqq \frac{b-a}{\sqrt2 \arctan\frac{\ln b-\ln a}{\sqrt2}},\\
	M_3(a,b)&\coloneqq \frac{\abs{b-a}}{2 \arctan \bigl(
			\abs{\ln b-\ln a}+1
		\bigr)-\frac\pi2},\\
	M_4(a,b)&\coloneqq \frac{b-a}{\sqrt2 \arcsinh\frac{\ln b-\ln a}{\sqrt2}},\\
	M_5(a,b)&\coloneqq \frac{\abs{b-a}}{\sqrt2 \Bigl(
		\arcsinh \bigl(
			1+\abs{\ln b-\ln a}
		\bigr)
		-\arcsinh1
	\Bigr)},\\
	M_{\alpha,r}(a,b)&\coloneqq \frac{(r+\alpha)\abs{b-a}}{\bigl(
			1+r\abs{\ln b-\ln a}
		\bigr)^{\frac{r+\alpha}{r}}-1}.
\end{aligned}
\ee

The (double) parameterized mean $M_{\alpha,r}$ includes some of the known classical means as well as the other means that appear to be new. The following result, which is a simple exercise of Real Analysis, clarifies this claim.

\begin{prop}\label{prop-Malphar-specialcases}
Let $r>0$ and $\abs{\alpha}\leq1$. Then the following statements hold:
\begin{enumerate}[(i)]
	\item $M_{\alpha,0}(a,b)=\frac{\alpha \abs{b-a}}{\exp\big({\alpha \abs{\ln b-\ln a}}\big)-1}$, with $M_{0,0}=L$.
	\item $M_{0,r}=L$ for any $r>0$, and $M_{\alpha,\infty}=L$ for any $\abs{\alpha}\leq1$.
	\item $M_{-r,r}(a,b)=\frac{r \abs{b-a}}{\ln\big(1+r \abs{\ln b-\ln a}\big)}$ for $0<r\leq1$. In particular, $M_{-1,1}=M_1$.
\end{enumerate}
\end{prop}

Let $M$ be a (bivariate) mean. For $a,b>0$ we can set $a=e^{-x}G$ and $b=e^xG$ for some $x\in{\mathbb R}$, where $G\coloneqq G(a,b)=\sqrt{ab}$. If $M$ is homogeneous, then
	\be\label{def-fM}
		M(a,b)=G\;M(e^{-x},e^x)\eqqcolon G\;f_M(x).
	\ee
If moreover $M$ is symmetric we can assume that $x\geq0$.
Relying on the associated functions $f_{m_i}$ we may express the characterization of the comparability condition $m_1<m_2$ for any two symmetric homogeneous means $m_1$ and $m_2$.
Namely, it is obvious that
	\begin{equation}\label{mf}
		 m_1(a,b)<m_2(a,b),\,\, \forall a\neq b
		 \ \Longleftrightarrow \
		 f_{m_1}(x)<f_{m_2}(x),\,\,\forall x>0.
	\end{equation}

For example, for the standard means, we have
\begin{gather}
	f_G(x)=1,\quad f_A(x)=\cosh x,\quad f_H(x)=\frac{1}{\cosh x},\quad f_L(x)=\frac{\sinh x}{x},\notag\\
	f_{P}(x)=\frac{\sinh x}{2\arctan\big(\tanh(x/2)\big)},\quad f_T(x)=
\frac{\sinh x}{\arctan(\tanh x)}, \label{f_}
\end{gather}
from which we easily deduce the well-known chain of inequalities $H<G<L<P<A<T$.

For the means $L_\alpha$ and $S_\alpha$, defined by \eqref{mean-def-Lalpha} and \eqref{mean-def-Salpha}, we have for any $x\neq0$
	\be\label{LS}
	f_{L_\alpha}(x)=4\alpha\frac{\sinh x}{\sinh(2\alpha x)},\quad f_{S_\alpha}(x)=\alpha\frac{\sinh x}{\arctan(\tanh\alpha x)},
	\ee
which implies that $L_{\alpha}<S_{\alpha}$ for any $0<\alpha\leq1$.

For the means \eqref{mean-def-Mi}, we can easily find that (for $x>0$)
\be\label{M}
\begin{aligned}
	f_{M_1}(x)&=\frac{2\sinh x}{\ln(1+2x)},
	& f_{M_2}(x)&=\frac{\sqrt{2}\sinh x}{\arctan(x\sqrt{2})},\\
	f_{M_3}(x)&=\frac{\sinh x}{\arctan(1+2x)-\pi/4},
	& f_{M_4}(x)&=\frac{\sqrt{2}\sinh x}{\arcsinh(x\sqrt{2})},\\
	f_{M_5}(x) &=\frac{\sqrt{2}\sinh x}{\arcsinh(1+2x)-\arcsinh1},
	& f_{M_{\alpha,r}}(x) &=\frac{2(r+\alpha)\sinh x}{\big(1+2rx\big)^{\frac{r+\alpha}{r}}-1}.
\end{aligned}
\ee

The rest of the paper is organized as follows. In Section 2, we establish comparison relations involving means from \cite{RaissRezgui-2019} and some other well-known means.
We examine the possibility of being $(A,G)$- or $(G,A)$-sub/super-stabilizable for means defined in \eqref{mean-def-Lalpha}, \eqref{mean-def-Salpha} and \eqref{mean-def-Mi}.
In Section 3, we present complete asymptotic expansions of the above mentioned means. We also extend the result from \cite{Mih-2023-sssa} in order to find the complete asymptotic expansion of the resultant mean-map of means whose asymptotic expansion may include all terms $x^{1-n}$, $n\in\bN_0$, which then could be applied on means which are subject of study in this paper. As a consequence of the coefficient comparison, we find parameters for which means $L_{\alpha}$ are stable and disprove the stability for other means. With use of the coefficients in the asymptotic expansion of power means (\cite{ElVu-2014-04}) we present the coefficients in the asymptotic expansion of the resultant mean-map $\cR(B_p,M,B_q)(x-t,x+t)$ as $x\to\infty$.
In Section 4, we show some of the applications of the obtained results. We analyze the behaviour of the difference $M-\cR(B_p,M,B_q)$, for each of the means defined in \eqref{mean-def-Lalpha}, \eqref{mean-def-Salpha} and \eqref{mean-def-Mi}. We examine when each of these means can be $(B_p,B_q)$-sub/super-stabilizable and, when possible, how to find such optimal parameters $p$ and $q$.

\section{Comparison of means and sub/super-stabilizability}

We start this section by stating some results about comparison between the bivariate means mentioned in the previous section.

\begin{prop}\label{prop-com}
Let $\alpha\in[0,1]$. The following statements hold:
\begin{enumerate}[(i)]
	\item \label{prop-com-a} If $0<\alpha<1/2$ then $G<L_{\alpha}<L<S_{\alpha}<A$.
	\item \label{prop-com-b} If $1/2<\alpha<1$ then $H<L_{\alpha}<G<L<S_{\alpha}<T$.
	\item \label{prop-com-c} If $1/2\leq\alpha\leq\sqrt{2}/2$ then $S_{\alpha}<A$.
\end{enumerate}
\end{prop}
\begin{proof}
For proving (\ref{prop-com-a}) and (\ref{prop-com-b}) we use the statement (\ref{prop-pLS-b}) of Proposition \ref{pLS} with the help of (\ref{prop-pLS-a}). The details are straightforward and therefore are omitted here.

To show (\ref{prop-com-c}) we use \eqref{mf} 
with $f_A$ defined in \eqref{f_} and $f_{S_{\alpha}}$ defined in the second formula in \eqref{LS}. Thus, we have to establish that
	$$
		\alpha\frac{\sinh x}{\arctan(\tanh \alpha x)}<\cosh x, \quad \forall x>0,
	$$
or equivalently,
	$$
		g(x)\coloneqq \alpha\tanh x-\arctan(\tanh \alpha x)<0,\quad \forall x>0.
	$$
Simple computation leads to
	\begin{align*}
		g{'}(x)&=\frac{\alpha}{(\cosh x)^2}-\frac{1}{1+(\tanh \alpha x)^2}\frac{\alpha}{(\cosh \alpha x)^2}\\
			&=\frac{\alpha}{(\cosh x)^2}-\frac{\alpha}{(\cosh \alpha x)^2+(\sinh \alpha x)^2}.
	\end{align*}
We need to study the sign of
	$$
		h(x)\coloneqq (\cosh \alpha x)^2+(\sinh \alpha x)^2-(\cosh x)^2,
	$$
for which we have
	$$
		h{'}(x)=2\alpha(\sinh2\alpha x)-\sinh2x\;\; \mbox{and}\;\; h{''}(x)=2\big(2\alpha^2\cosh2\alpha x-\cosh2x\big).
	$$
If $\alpha\leq\sqrt{2}/2$, it is easy to see that $h{''}(x)<0$ for all $x>0$. Thus, $h{'}$ is strictly decreasing for $x>0$ and so, $h{'}(x)<h{'}(0)=0$ for all $x>0$. By the same arguments, we deduce that $h(x)<h(0)=0$ and therefore $g(x)<g(0)=0$, for all $x>0$. The proof is finished.
\end{proof}

\begin{rem}
Numerical computations show that if $\sqrt{2}/2<\alpha<1$ then $S_\alpha$ and $A$ are not comparable.
\end{rem}

The following Proposition is an extension of the result from \cite{RaissRezgui-2019}.

\begin{prop}\label{prop-comM}
We have the following assertions:
\begin{enumerate}[(i)]
	\item \label{prop-comM-a} $M_4<M_5<M_1<M_3$ and $L<M_4<A$. The mean $A$ is not comparable to either one of $M_1, M_2, M_3$ and $M_5$.
	\item \label{prop-comM-b} $L<M_2<M_3$. The mean $M_2$ is not comparable to either one of $M_1,M_4$ and $M_5$.
\end{enumerate}
\end{prop}
\begin{proof}
	 To prove $M_4<M_5<M_1<M_3$ in (\ref{prop-comM-a}), relying on  \eqref{mf} and \eqref{M}, it is equivalent to show that for all $x>0$ the following chain of inequalities holds
	\begin{multline}\label{CH}
		\frac{\sqrt{2}\sinh x}{\arcsinh(x\sqrt{2})}<\frac{\sqrt{2}\sinh x}{\arcsinh(1+2x)-\arcsinh1}\\
		<\frac{2\sinh x}{\ln(1+2x)}<\frac{\sinh x}{\arctan(1+2x)-\pi/4}.
	\end{multline}
	
For the first inequality of \eqref{CH} we consider
	$$
		g(x)\coloneqq \arcsinh(1+2x)-\arcsinh(x\sqrt{2})-\arcsinh1,
	$$
and then
	$$
		g{'}(x)=\frac{2}{\sqrt{1+(1+2x)^2}}-\frac{\sqrt{2}}{\sqrt{1+2x^2}},
	$$
for which it is easy to see that $g{'}(x)<0$ and so $g(x)<g(0)= 
0$, for all $x>0$. 

The proof of the two other inequalities in \eqref{CH} as well as the proof of the inequalities $L<M_4<A$ and also $L<M_2<M_3$ from part (\ref{prop-comM-b}) is similar. The details are therefore omitted here.

To show that, for example, $A$ is not comparable with $M_1$ we proceed as follows: we compare $f_A(x)$ and $f_{M_1}(x)$ for $x>0$, or equivalently, we study the sign of $g(x)\coloneqq \ln(1+2x)-2\tanh x$. It is easy to check that $\lim_{x\uparrow\infty}g(x)=+\infty$ and $g(2)=\ln 5-2\tanh 2<0$. We then have the conclusion.
\end{proof}

The following result concerns comparison of $M_{\alpha,r}$ with $A$ and $G$.

\begin{prop}
Let $r>0$. If $\alpha\leq0$ then $M_{\alpha,r}>G$ and, if $\alpha\geq0$ then $M_{\alpha,r}<A$.
\end{prop}
\begin{proof}
Assume that $\alpha\leq0$. We want to show that $f_{M_{\alpha,r}}(x)>f_G(x)$ for any $x>0$. We consider
$$g(x)\coloneqq \frac{h(x)}{\big(1+2rx\big)^{\frac{\alpha+r}{r}}-1},$$
with
$$h(x)=2(\alpha+r)\sinh x-\big(1+2rx\big)^{\frac{\alpha+r}{r}}+1.$$
It is clear that $h{'}(x)=2(r+\alpha)\Big(\cosh x-\big(1+2rx\big)^{\frac{\alpha}{r}}\Big)$. We have the following situations:
\begin{itemize}
	\item If $\alpha+r>0$ then $\big(1+2rx\big)^{\frac{\alpha+r}{r}}-1>0$ and if moreover $\alpha\leq0$ then $\cosh x>1\geq(1+2rx)^{\alpha/r}$ for any $x>0$. In this case, $h{'}(x)>0$ and so $h(x)>h(0)=0$, for all $x>0$. We then deduce that $g(x)>0$ for all $x>0$.
	\item If $\alpha+r<0$ then $\big(1+2rx\big)^{\frac{\alpha+r}{r}}-1<0$ and if moreover $\alpha\leq0$ then $h{'}(x)<0$ and so $h(x)<h(0)=0$, for all $x>0$. We then infer that $g(x)>0$ for any $x>0$.
\end{itemize}
Summarizing, we have shown the desired result.

The inequality $M_{\alpha,r}<A$, for $\alpha\geq0$, can be established in a similar way, and we leave it to the reader.
\end{proof}

\begin{rem}
Numerical computations show that, if $\alpha>0$ (resp.\ $\alpha<0$) then $M_{\alpha,r}$ is not comparable with $G$ (resp.\ $A$).
\end{rem}

We will now study the sub/super-stabilizability of some of the above means. We recall the following result. 

\begin{prop}[\cite{RaissSandor-2014}]\label{prop-rs}
Let $M$ be a continuous symmetric mean. Then
\begin{enumerate}[(i)]
	\item If $M$ is $(A,G)$-sub-stabilizable then $L\leq M\leq A$.
	\item If $M$ is $(A,G)$-super-stabilizable then $G\leq M\leq L$.
\end{enumerate}
\end{prop}

Combining Proposition \ref{prop-rs} with Proposition \ref{prop-com} and Proposition \ref{prop-comM} we immediately deduce the following corollary.

\begin{cor}
Let $\alpha\in[0,1]$. The following statements hold:
\begin{enumerate}[(i)]
	\item  $L_{\alpha}$ is not $(A,G)$-sub-stabilizable and $S_{\alpha}$ is not $(A,G)$-super-stabilizable.
	\item If $1/2<\alpha\leq1$ then $L_\alpha$ is not $(A,G)$-super-stabilizable.
	\item  The three means $M_1,M_3$ and $M_5$ are neither $(A,G)$-sub/super-stabilizable nor $(G,A)$-sub/super-stabilizable.
\end{enumerate}
\end{cor}

To giving more results about sub/super-stabilizability we need the following lemma.

\begin{lemma}[\cite{Raiss-2011-stab}]\label{lemma-AG}
Let $m$ be a symmetric homogeneous mean. Then, for any $a,b>0$, we have
$$\cR(A,m,G)(a,b)=A\big(\sqrt{a},\sqrt{b}\big)\;m\big(\sqrt{a},\sqrt{b}\big).$$
\end{lemma}

Now, we may state the following result.

\begin{thm}
Let $0<\alpha\leq\sqrt{2}/2$. Then, $S_{\alpha}$ is strictly $(A,G)$-sub-stabilizable. In particular, the first Seiffert mean $P=S_{1/2}$ is strictly $(A,G)$-sub-stabilizable.
\end{thm}
\begin{proof}
	Firstly, by following Proposition \ref{prop-com} we have $G\leq S_{\alpha}\leq A$ for any $0<\alpha\leq\sqrt{2}/2$. According to Definition \ref{defsubsuperst} we have to prove that the inequality $\cR\big(A,S_{\alpha},G\big)(a,b)<S_{\alpha}(a,b)$ holds for any $a>b$. By Lemma \ref{lemma-AG}, with substitution \eqref{def-fM}, it is equivalent to show that the inequality
	\bes
		f_A(\tfrac{x}2) f_{S_{\alpha}}(\tfrac{x}2) < f_{S_{\alpha}}(x)
	\ees
	holds for any $x>0$,
	and in combination with the corresponding relations in \eqref{f_} and in \eqref{LS}, 
	this is equivalent to 
	$$
		\alpha\cosh(x/2)\frac{\sinh(x/2)}{\arctan\big(\tanh(\alpha x/2)\big)}		
			<\alpha\frac{\sinh x}{\arctan\big(\tanh(\alpha x)\big)},
	$$
for $x>0$. 
Using formula $\sinh x=2\sinh(x/2)\cosh(x/2)$ and setting $g(x)\coloneqq \arctan\big(\tanh(\alpha x)\big)-2\arctan\big(\tanh(\alpha x/2)\big)$ we easily verify that $g{'}(x)<0$ and so $g(x)<g(0)=0$, for all $x>0$. We then deduce the desired result.
\end{proof}

\begin{thm}
	Let $0<\alpha<1/2$. Then $L_\alpha$ is strictly $(A,G)$-super-stabilizable. In particular, the Heinz mean $HZ_{1/4}=L_{1/4}$ is strictly $(A,G)$-super-stabilizable.
\end{thm}
\begin{proof}
	By similar way and similar arguments as in the proof of the previous theorem, we are here reduced to study the sign of $g(x)\coloneqq \sinh(2\alpha x)-2\sinh(\alpha x)$ for $x>0$. Obviously, $g{'}(x)\coloneqq 2\alpha\cosh(2\alpha x)-2\alpha\cosh(\alpha x)>0$ and so $g(x)>g(0)=0$, for any $x>0$ so concluding the proof.
\end{proof}

\begin{thm}
	The means $M_2$ and $M_4$ are both strictly $(A,G)$-sub-stabilizable.
\end{thm}
\begin{proof}
	It is also similar to the previous proofs. The details are straightforward and therefore omitted here.
\end{proof}

\section{Asymptotic analysis of new means}




\subsection{Asymptotic expansions of means from \cite{RaissRezgui-2019}}

Let us find asymptotic expansions of means from the Section 1.

The most used result is the expansion for the power of an asymptotic series, which we recall here.

\begin{lemma}[\cite{ChenElVu-2013}]
	\label{lemma-power}
	Let
	$$
		g(x)\sim \sum_{n=0}^\infty a_n x^{-n}
	$$
	be a given asymptotic expansion (for $x\to\infty$) of $g(x)$ with $a_0\neq0$. Then for all real $r$ it holds
	$$
		[g(x)]^{r}\sim \sum_{n=0}^\infty P[n,r,\ba]x^{-n},
	$$
	where $P[0,r,\ba]=a_0^r$ and
	\bes
	        P[n,r,\ba]=\frac1{na_0}\sum_{k=1}^n[k(1+r)-n]a_kP[n-k,r,\ba],\quad n\in\bN.
	\ees
\end{lemma}

We assume all sequences are enumerated from 0.
Here $P[n,r,\ba]$ denotes the coefficient by the $x^{-n}$ in the $r$-th power of series assigned to a sequence $\ba=(a_i)_{i\in\bN_0}$.

\begin{prop}\label{prop-coeff-Lalpha}
Complete asymptotic expansion of mean $L_\alpha$, $\abs{\alpha}\le1$, is given by:
\bes
	L_\alpha (x-t,x+t)\sim x\sum_{n=0}^\infty
		2\alpha\sum_{k=0}^n
		\binom{\alpha}{n-k}(-1)^{n-k} P[k,-1,(\tbinom{2\alpha}{2i+1})_{i\in\bN_0}] t^{2n} x^{-2n}.
\ees
\end{prop}
\begin{proof}
	\begin{align*}
		L_\alpha (x-t,x+t) &= \frac{4\alpha t (x-t)^\alpha (x+t)^\alpha}{(x+t)^{2\alpha}-(x-t)^{2\alpha}}\\
			&=4\alpha t \left(1-\frac{t}x\right)^\alpha \left(1+\frac{t}x\right)^\alpha
				\left[ \left(1+\frac{t}x\right)^{2\alpha}- \left(1-\frac{t}x\right)^{2\alpha}\right]^{-1} \\
			&\sim 4\alpha t \left(1-\frac{t^2}{x^2}\right)^\alpha
				\left[ \sum_{k=0}^\infty \binom{2\alpha}{k} t^k x^{-k}
					-\sum_{k=0}^\infty \binom{2\alpha}{k}(-1)^k t^k x^{-k} \right]^{-1}\\
			&\sim 2\alpha t \left(1-\frac{t^2}{x^2}\right)^\alpha
				\left[ \sum_{k=0}^\infty \binom{2\alpha}{2k+1} t^{2k+1} x^{-(2k+1)}
					 \right]^{-1}\\
			&\sim 2\alpha t \sum_{j=0}^\infty  \binom{\alpha}{j}(-1)^j t^{2j} x^{-2j}
				\cdot \frac{x}{t} \sum_{k=0}^\infty
					P[k,-1,( \tbinom{2\alpha}{2i+1} )_{i\in\bN_0} ]
					 t^{2k} x^{-2k}\\
			&\sim 2\alpha x \sum_{n=0}^\infty \sum_{k=0}^n \binom{\alpha}{n-k}(-1)^{n-k}
				P[k,-1,( \tbinom{2\alpha}{2i+1} )_{i\in\bN_0} ] t^{2n} x^{-2n}.
	\end{align*}
\end{proof}

The beginning of the asymptotic expansion:
\be\label{exp-Mu-cosh}
\begin{aligned}
	L_\alpha (x&-t,x+t)\sim  x-\tfrac13(2\alpha^2+1) t^2 x^{-1}
		+\tfrac{2}{45}(\alpha-1)(\alpha+1)(7\alpha^2+2) t^4 x^{-3}\\
		&\quad-\tfrac{2}{945}(\alpha-1)(\alpha+1)(62\alpha^4-85\alpha^2-22) t^6 x^{-5}\\
		&\quad+\tfrac{2}{14175}(\alpha-1)(\alpha+1)(381\alpha^6-1169\alpha^4+889\alpha^2+214)t^8 x^{-7}
		+\dots
\end{aligned}
\ee
For $\alpha=1$ and $\alpha=\frac12$ these coefficients coincide with coefficients obtained in paper \cite{ElVu-2014-09} for harmonic and geometric mean respectively.

\begin{prop}\label{prop-coeff-Salpha}
Complete asymptotic expansion of mean $S_\alpha$, $\abs{\alpha}\le1$, is given by:
\bes
	S_\alpha (x-t,x+t)\sim x\sum_{n=0}^\infty
		 \alpha P[n,-1,(D_i)_{i\in\bN_0}] t^{2n} x^{-2n},
\ees
where
\bes
	D_n=\sum_{m=0}^n \frac{(-1)^m}{2m+1} P[n-m,2m+1,(C_i)_{i\in\bN_0}],
\ees
and
\bes
	C_n=\sum_{k=0}^n\binom{\alpha}{2k+1}P[n-k,-1,(\tbinom{\alpha}{2i})_{i\in\bN_0}].
\ees
\end{prop}
\begin{proof}
	\begin{align*}
		S_\alpha &(x-t,x+t)
		= \alpha t \left[ \arctan
			\frac{(x+t)^\alpha-(x-t)^\alpha}{(x+t)^\alpha+(x-t)^\alpha} \right]^{-1}\\
		&= \alpha t \left[ \arctan
			\frac{\left(1+\frac{t}{x}\right)^\alpha-\left(1-\frac{t}{x}\right)^\alpha}{\left(1+\frac{t}{x}\right)^\alpha + \left(1-\frac{t}{x}\right)} \right]^{-1}\\
		&\sim \alpha t \left[ \arctan \frac{\sum_{k=0}^\infty\binom{\alpha}{k}t^{k}x^{-k}-\sum_{k=0}^\infty\binom{\alpha}{k} (-1)^k t^{k}x^{-k}}{\sum_{k=0}^\infty\binom{\alpha}{k}t^{k}x^{-k}+\sum_{k=0}^\infty\binom{\alpha}{k} (-1)^k t^{k}x^{-k}} \right]^{-1}\\
		&\sim \alpha t \left[ \arctan \frac{\sum_{k=0}^\infty\binom{\alpha}{2k+1}t^{2k+1}x^{-(2k+1)}}{\sum_{k=0}^\infty\binom{\alpha}{2k}t^{2k}x^{-2k}} \right]^{-1}\\
		&\sim  \alpha t \left[ \arctan \left(
			\frac{t}{x}\sum_{k=0}^\infty\binom{\alpha}{2k+1}t^{2k}x^{-2k}
			\cdot \sum_{j=0}^\infty P[j,-1,(\tbinom{\alpha}{2i})_{i\in\bN_0}]t^{2j}x^{-2j}
			\right) \right]^{-1}\\
		&\sim  \alpha t \Biggl[ \arctan \biggl(
			\frac{t}{x} \sum_{n=0}^\infty \underbrace{\sum_{k=0}^n \binom{\alpha}{2k+1}
				P[n-k,-1,(\tbinom{\alpha}{2i})_{i\in\bN_0}]}_{C_n} t^{2n}x^{-2n}
			\biggr) \Biggr]^{-1}\\
		&\sim \alpha t\left[
			\sum_{m=0}^\infty \frac{(-1)^m}{2m+1} \left(
				\frac{t}{x} \sum_{n=0}^\infty C_n t^{2n} x^{-2n}
				\right)^{2m+1}
			\right]^{-1}\\
		&\sim \alpha t\left[
			\sum_{m=0}^\infty \frac{(-1)^m}{2m+1} \left(\frac{t}{x}\right)^{2m+1}
				\sum_{j=0}^\infty P[j,2m+1,(C_i)_{i\in\bN_0}] t^{2j}x^{-2j}
			\right]^{-1}\\
		&\sim \alpha t\Biggl[ \frac{t}{x} \sum_{n=0}^\infty \underbrace{\sum_{m=0}^n
			\frac{(-1)^m}{2m+1} P[n-m,2m+1,(C_i)_{i\in\bN_0}]}_{D_n} t^{2n} x^{-2n}
			\Biggr]^{-1}\\
		&\sim \alpha x \sum_{n=0}^\infty P[n,-1,(D_i)_{i\in\bN_0}] t^{2n} x^{-2n}.
	\end{align*}
\end{proof}

Although computed using several recursively defined sequences, the coefficients have a nice form. The beginning of the asymptotic expansion is given by
\be\label{exp-Mu-1cosh}
\begin{aligned}
	S_\alpha &(x-t,x+t)\sim  x+\tfrac13(2\alpha^2-1) t^2 x^{-1}
		-\tfrac{2}{45}(5\alpha^4-5\alpha^2+2) t^4 x^{-3}\\
		&\quad+\tfrac{2}{945}(86\alpha^6-105\alpha^4+63\alpha^2-22) t^6 x^{-5}\\
		&\quad-\tfrac{2}{14175}\bigl(214+5\alpha^2(\alpha-1)(\alpha+1)(135-159\alpha^2+271\alpha^4)\bigr)t^8 x^{-7}
		+\dots
\end{aligned}
\ee
For $\alpha=\frac12$ and $\alpha=1$ these coefficients coincide with the coefficients from \cite{Vu} for the first and the second Seiffert mean.

We may state more general result for means of the type \eqref{mean-type-mg}.
\begin{thm}\label{thm-exp-mu-mg}
Assume that the odd function $\mu\colon\bR\to\bR$ has the following expansion
\be\label{thm-mu-assumptions}
	\mu(x)\sim  \sum_{n=0}^\infty c_n x^{2n+1},\ \text{as}\ x\to0,
\ee
with $c_0=1$. Then mean $m_g$ defined in \eqref{mean-type-mg} has the following expansion
\be\label{thm-exp-mg}
	m_g(x-t,x+t)\sim x\sum_{m=0}^\infty P[m,-1,(E_i)_{i\in\bN_0}]t^{2m} x^{-2m},
	\ \text{as}\ x\to\infty,
\ee
where
\bes
	E_m=\sum_{n=0}^m c_n 2^{2n} P[m-n,2n+1, (\tfrac1{2i+1})_{i\in\bN_0}].
\ees
\end{thm}
\begin{proof}
	Observe the expression $\lvert \ln \frac{a}{b} \rvert$ when $a=x-t$ and $b=x+t$,
	which under assumption $t>0$ and for $x$ large enough, is equal to
	\begin{align*}
		\ln\frac{x+t}{x-t}
			&\sim\ln \left(1+\frac{t}{x}\right)-\ln \left(1-\frac{t}{x}\right)
			=\sum_{k=1}^\infty\frac{(-1)^{k+1}}{k} t^k x^{-k}
			+\sum_{k=1}^\infty\frac{1}{k} t^k x^{-k}\\
			&\sim 2\sum_{k=0}^\infty\frac{1}{2k+1} t^{2k+1} x^{-(2k+1)}.
	\end{align*}
	Now we have, for any $t$,
	\begin{align*}
		m_g&(x-t,x+t)\sim 2t \left[\mu \left(
			2\sum_{k=0}^\infty\frac{1}{2k+1} t^{2k+1} x^{-(2k+1)}
			\right) \right]^{-1} \\
			&\sim 2t \left[ \sum_{n=0}^\infty c_n \left(
				2 \sum_{k=0}^\infty\frac{1}{2k+1} t^{2k+1} x^{-(2k+1)}
				\right)^{2n+1} \right]^{-1} \\
			&\sim 2t \left[ \sum_{n=0}^\infty c_n t^{2n+1} x^{-2n-1} 2^{2n+1}
				\sum_{k=0}^\infty P[k,2n+1,(\tfrac1{2i+1})_{i\in\bN_0}] t^{2k}x^{-2k}
				\right]^{-1} \\
			&\sim 2t \Biggl[ 2 t x^{-1} \sum_{m=0}^\infty
				\underbrace{\sum_{n=0}^m c_n 2^{2n} P[m-n,2n+1,(\tfrac1{2i+1})_{i\in\bN_0}]}_{E_m}
				 t^{2m} x^{-2m}
				\Biggr]^{-1} \\
			&\sim x \sum_{m=0}^\infty P[m,-1,(E_i)_{i\in\bN_0}] t^{2m} x^{-2m}.
	\end{align*}
\end{proof}

\begin{exa}
	\begin{enumerate}
		\item Asymptotic expansion of mean $L_\alpha$ can also be deduced form Theorem \ref{thm-exp-mu-mg}. Let $\mu(x)=\frac1\alpha \sinh(\alpha x)$, i.e.\ let $c_n=\frac{\alpha^{2n}}{(2n+1)!}$ in \eqref{thm-mu-assumptions}. Then using formula \eqref{thm-exp-mg} we may also obtain the coefficients from \eqref{exp-Mu-cosh}.
		\item Let $u(y)=\frac1{\cosh(\alpha y)}$ whose asymptotic expansion as $y\to0$ is
		$u(y)\sim\sum_{n=0}^\infty P[n,-1,(\tfrac{\alpha^{2i}}{(2i)!})_{i\in\bN_0}] y^{2n}$. Then, with $c_n=\frac1{2n+1}P[n,-1,(\tfrac{\alpha^{2i}}{(2i)!})_{i\in\bN_0}]$
		in \eqref{thm-mu-assumptions}, formula \eqref{thm-exp-mg} gives the coefficients as in \eqref{exp-Mu-1cosh}.
	\end{enumerate}
\end{exa}

Regarding the rest of the means from Section 1, similar computations lead to their asymptotic expansions. With prior use of the arctangent addition formula in the denominator of $M_3$ and the inverse hyperbolic sine addition (subtraction) formula for the denominator of $M_5$, with application of Lemma \ref{lemma-power} and Taylor series expansion of logarithmic, arctangent or inverse hyperbolic sine function in variable $t/x$,
we obtain the following expansions.

\begin{prop}\label{prop-exp-Mi}
As $x\to \infty$, for means defined in \eqref{mean-def-Mi}, the following expansions hold:
\begin{align*}
	M_1 (x-t,x+t)&\sim
			x
			+\abs{t}
			-\tfrac23 t^2 x^{-1}
			+\tfrac13 \abs{t}^3 x^{-2}
			-\tfrac{28}{45} t^4 x^{-3}
			+\tfrac{37}{45} \abs{t}^5 x^{-4}\\
			&\qquad-\tfrac{1369}{945} t^6 x^{-5}
		+\dots\\
	M_2 (x-t,x+t)&\sim
			x
			+\tfrac13 t^2 x^{-1}
			-\tfrac29 t^4 x^{-3}
			+\tfrac{14}{135} t^6 x^{-5}
			-\tfrac{122}{945} t^8 x^{-7}
		+\dots\\
	M_3(x-t,x+t) &\sim
			x
			+\abs{t}
			-\tfrac13 \abs{t}^3 x^{-2}
			+\tfrac{4}{15} t^4 x^{-3}
			-\tfrac{13}{45} \abs{t}^5 x^{-4}
			+\tfrac19 t^6 x^{-5}
		+\dots\\
	M_4(x-t,x+t)&\sim
			x
			-\tfrac16 t^4 x^{-3}
			+\tfrac{8}{315} t^6 x^{-5}
			-\tfrac{367}{4536} t^8 x^{-7}
		+\dots\\
	M_5(x-t,x+t) &\sim
			x
			+\tfrac12 \abs{t}
			-\tfrac14 t^2 x^{-1}
			-\tfrac16 \abs{t}^3 x^{-2}
			+\tfrac5{48} t^4 x^{-3}
			-\tfrac73{360} \abs{t}^5 x^{-4}\\
			&\qquad+\tfrac{1033}{10 080} t^6 x^{-5}
		+\dots\\
	M_{\alpha,r}(x-t,x+t) &\sim
			x
			+\alpha \abs{t}
			+\tfrac13(\alpha^2+2r\alpha-1) t^2 x^{-1}
			-\tfrac13 r\alpha(2r+\alpha) \abs{t}^3 x^{-2}\\
			&\qquad
			-\tfrac{1}{45} \left(\alpha  (\alpha +2 r) \left(\alpha ^2-18 r^2+2 \alpha  r-5\right)+4\right) t^4 x^{-3}\\
			&\qquad -\tfrac{1}{45} \alpha  r (\alpha +2 r)
				\bigl(3 (2 r-\alpha ) (\alpha +4 r)+10\bigr) \abs{t}^5 x^{-4}
		+\dots
\end{align*}
\end{prop}

\subsection{Asymptotic expansion of the resultant mean-map}

Assuming that all means involved were bivariate, symmetric, homogeneous and had the asymptotic expansions as $x\to\infty$ of the following type
\begin{align}
	K(x-t,x+t) &\sim \sum_{n=0}^\infty a^K_n t^{2n}x^{-2n+1},\notag\\
	M(x-t,x+t) &\sim \sum_{n=0}^\infty a^M_n t^{2n}x^{-2n+1},\label{def-MNK-M}\\
	N(x-t,x+t) &\sim \sum_{n=0}^\infty a^N_n t^{2n}x^{-2n+1},\notag
\end{align}
in \cite{Mih-2023-sssa} we found the complete asymptotic expansion of the resultant mean--map:
\be\label{exp-resultant-evenoddpowers}
	R(x-t,x+t) = \cR(K,M,N)(x-t,x+t) \sim \sum_{n=0}^\infty a^R_n t^{2n}x^{-2n+1}.
\ee
Using coefficients of the resultant mean--map of the corresponding means, we found formula for calculating the coefficients in the asymptotic power series expansion of type \eqref{def-MNK-M} of stable mean, which, up to five terms reads as (\cite[formula (34)]{Mih-2023-sssa}):
\be\label{exp-Mtt-stable}
\begin{aligned}
	&M (x-t,x+t)= x+ a_1 t^2x^{-1}
		+  \tfrac1{6}a_1(1+a_1)(1-4a_1) t^4x^{-3}  \\
		&\quad+ \tfrac1{90} a_1(1+a_1)\left(6-31a_1+36a_1^2+64a_1^3\right) t^6x^{-5}  \\
		&\quad+ \tfrac1{2520} a_1 (1+a_1) \bigl(90 - 531 a_1  + 937 a_1 ^2 + 568 a_1 ^3 - 3088 a_1 ^4 - 2176 a_1^5 \bigr) t^8x^{-7}  \\
		&\quad +\O(x^{-9}).
\end{aligned}
\ee

With $K=B_p$ and $N=B_q$, we obtain coefficients in the asymptotic expansion \eqref{exp-resultant-evenoddpowers} of the resultant mean--map (consequence of \cite[formula (26)]{Mih-2023-sssa}):
 	\be\label{coeff-RBpNBq}
 	\begin{aligned}
 		&\cR(B_p, M,B_q)(x-t,x+t) = x + \tfrac{1}{8} (2 a_1^M+p+2q -3) t^2 x^{-1} \\
 			& \quad + \tfrac{1}{384} \Bigl(24 a_2^M+
 			12 a_1^M (-4 p q+p+2 q (q+1)+1) \\
			&\qquad -2 p^3+3 p^2+2 p (7-6 q)+4 q \left(-4 q^2+6 q+7\right)-39
 		\Bigr) t^4 x^{-3} +\O(x^{-5}).		
 	\end{aligned}
 	\ee
 	
As a consequence of the results of the above mentioned paper, specially relying on the form of the stable mean coefficients given in \eqref{exp-Mtt-stable}, in combination with the asymptotic expansion given in Proposition \ref{prop-coeff-Lalpha} we obtain the following.
\begin{cor}\label{cor-Lalpha-stable}
	Means $L_\alpha$ are stable iff $\alpha\in\{\pm\frac12,\pm1\}$.
\end{cor}
\begin{proof}
	By comparing coefficients \eqref{exp-Mu-cosh} in the asymptotic expansion of mean $L_\alpha$ with the stable mean coefficients \eqref{exp-Mtt-stable} by the powers
	$x^{-1}$ and $x^{-3}$
	we obtain the following equations
	\begin{align*}
		  -\tfrac13(2\alpha^2+1) &= a_1,\\
		  \tfrac2{45}(\alpha-1)(\alpha+1)(7\alpha^2+2) &=\tfrac16 a_1(1+a_1)(1-4a_1)\\
			& =\tfrac16 (-\tfrac13(2\alpha^2+1)) (1-\tfrac13(2\alpha^2+1))
				(1+\tfrac43(2\alpha^2+1)),
	\end{align*}
	which combined give the following
	$$
		 18(\alpha^2-1)(7\alpha^2+2)=5(2\alpha^2+1)(\alpha^2-1)(8\alpha^2+7).
	$$	
	The only real solutions are $\alpha\in\{\pm\frac12,\pm1\}$. For either of those values of $\alpha$ the stable mean is obtained. Namely,
	 $L_{\pm1}=H$ and $L_{\pm\frac12}=G$.
\end{proof}

The similar procedure as in Corollary \ref{cor-Lalpha-stable}, involving coefficients given in \eqref{exp-Mu-1cosh}, gives no solutions for $\lvert\alpha\rvert\le1$, so we have the following conclusion.
\begin{cor}
	Means $S_\alpha$, $\lvert\alpha\rvert\le1$, are not stable.
\end{cor}

Not all of the means \eqref{mean-def-Lalpha}, \eqref{mean-def-Salpha}, \eqref{mean-def-Mi} have asymptotic expansion of the form \eqref{def-MNK-M}. In order to obtain the asymptotic expansion of the resultant mean--map  for all of the means mentioned in the introduction, and afterwards find which of them are stable and study asymptotic expansion of $\cR(B_p,M,B_q)$,
we need to adjust the algorithm obtained in the paper \cite{Mih-2023-sssa}.

\begin{thm} \label{thm-asymexp-resultant-all}
	Assume that means $K$, $M$ and $N$ have expansions
	\begin{align}
		K(x-t,x+t) &\sim \sum_{n=0}^\infty a^K_n t^{n}x^{-n+1},\label{def-MNK-Kt}\\
		M(x-t,x+t) &\sim \sum_{n=0}^\infty a^M_n t^{n}x^{-n+1},\label{def-MNK-Mt}\\
		N(x-t,x+t) &\sim \sum_{n=0}^\infty a^N_n t^{n}x^{-n+1},\label{def-MNK-Nt}
	\end{align}
	as $x\to\infty$, such that $a_1^N \neq \pm1$.
	Then the coefficients in the asymptotic expansion of the resultant mean--map
	\be\label{exp-resultant-allpowers}
		R(x-t,x+t) = \cR(K,M,N)(x-t,x+t) \sim \sum_{n=0}^\infty a^R_n t^{n} x^{-n+1},
	\ee
	can be calulated by the recursive formula
	\be\label{thm-asymexp-resultant-all-amR}
		a_m^R = \frac14 \sum_{n=0}^m \sum_{k=0}^{m-n} a_n^K
				P[k,n,\bd] P[m-n-k,-n+1,\bs],\quad m\in\bN_0,
	\ee
	where	
	\begin{align*}
		d_{m-1}&= \sum_{n=0}^m a_n^M \sum_{k=0}^{m-n}
			 \bigl( P[k,n,\tilde{\bg}] P[m-n-k,-n+1,\tilde{\bh}] \\
			 	&\qquad\qquad\qquad-P[k,n,\bg] P[m-n-k,-n+1,\bh] \bigr),\quad m\in\bN, \\
		s_m&= \sum_{n=0}^m a_n^M \sum_{k=0}^{m-n}
			 \bigl( P[k,n,\tilde{\bg}] P[m-n-k,-n+1,\tilde{\bh}] \\
			 	&\qquad \qquad\qquad+P[k,n,\bg] P[m-n-k,-n+1,\bh]\bigr),\quad m\in\bN_0,
	\end{align*}
	and 	
	\be\label{thm-def-bgbh}
	\begin{aligned}
		\bg &= (1+a_1^N, a_2^N, a_3^N,\ldots),
		&\bh &= (2,a_1^N-1, a_2^N, a_3^N,\ldots), \\
		\tilde{\bg} &= (1-a_1^N, - a_2^N, - a_3^N,\ldots),
		&\tilde{\bh} &= (2,1+a_1^N, a_2^N, a_3^N,\ldots).
	\end{aligned}
	\ee
\end{thm}

\begin{proof}
The proof goes by the similar procedure as the proof of the somewhat specific analogue form paper \cite{Mih-2023-sssa}. With $N$ being the abbreviated version of $N(x-t,x+t)$, the following holds:
	\begin{align*}
		&M(x-t,  N(x-t,x+t))
			 \sim \sum_{n=0}^\infty a_n^M \bigl(\tfrac12 (N-x+t)\bigr)^n
			 	\bigl(\tfrac12(N+x-t)\bigr)^{-n+1} \\
			& \sim \frac12 \sum_{n=0}^\infty a_n^M
				\Bigl(t \bigl(1+a_1^N\bigr)+\sum_{k=2}^\infty a_k^N t^k x^{-k+1}\Bigr)^n
				\Bigl(2x+\bigl(a_1^N-1\bigr)t+\sum_{j=2}^\infty a_j^N t^j x^{-j+1}\Bigr)^{-n+1}\\
			& \sim \frac12 \sum_{n=0}^\infty a_n^M
				 t^n \sum_{k=0}^\infty P[k,n,\bg] t^k x^{-k}
				 \cdot x^{-n+1} \sum_{j=0}^\infty P[j,-n+1,\bh] t^j x^{-j}\\
			&\sim \frac12 \sum_{m=0}^\infty \sum_{n=0}^m a_n^M
				\sum_{k=0}^{m-n} P[k,n,\bg] P[m-n-k,-n+1,\bh]
				t^m x^{-m+1}.
	\end{align*}
Similarly, the second component to be composed with the mean $K$ has the following expansion:
	\begin{align*}
		&M(N(x-t,x+t),x+t)
			 \sim \sum_{n=0}^\infty a_n^M \bigl(\tfrac12 (x+t-N)\bigr)^n
			 	\bigl(\tfrac12(x+t+N)\bigr)^{-n+1} \\
			& \sim \frac12 \sum_{n=0}^\infty a_n^M
				\Bigl(t \bigl(1-a_1^N\bigr)-\sum_{k=2}^\infty a_k^N t^k x^{-k+1}\Bigr)^n
				\Bigl(2x+\bigl(a_1^N+1\bigr)t+\sum_{j=2}^\infty a_j^N t^j x^{-j+1}\Bigr)^{-n+1}\\
			& \sim \frac12 \sum_{n=0}^\infty a_n^M
				 t^n \sum_{k=0}^\infty P[k,n,\tilde{\bg}] t^k x^{-k}
				 \cdot x^{-n+1} \sum_{j=0}^\infty P[j,-n+1,\tilde{\bh}] t^j x^{-j}\\
			&\sim \frac12 \sum_{m=0}^\infty \sum_{n=0}^m a_n^M
				\sum_{k=0}^{m-n} P[k,n,\tilde{\bg}] P[m-n-k,-n+1,\tilde{\bh}]
				t^m x^{-m+1}.
	\end{align*}
In order to obtain the desired formula, we need to calculate one half of the difference of the previous two expressions
	\bes
		T 	 = \frac12\bigl( M(N(x-t,x+t),x+t) - M(x-t, N(x-t,x+t)) \bigr) 
			  \sim \frac14  \sum_{m=1}^\infty d_{m-1} t^m x^{-m+1},  		 	
	\ees 
and one half of the sum of those two expressions
	\bes 
		X 	 = \frac12\bigl( M(N(x-t,x+t),x+t) + M(x-t,  N(x-t,x+t)) \bigr)\\
			  \sim \frac14  \sum_{m=0}^\infty s_m t^m x^{-m+1}. 		 	
	\ees 
The resultant mean-map can then be written in a following way	
	\begin{align*}
		R 	&= K(X-T,X+T) \\
			& \sim \sum_{n=0}^\infty a_n^K
				\Bigl( \tfrac14 \sum_{k=1}^\infty d_{k-1} t^k x^{-k+1} \Bigr)^n
				\Bigl( \tfrac14 \sum_{j=0}^\infty s_j t^j x^{-j+1} \Bigr)^{-n+1} \\
			& \sim \frac14 \sum_{n=0}^\infty a_n^K
				 t^n \sum_{k=0}^\infty P[k,n,\bd] t^k x^{-k}
				 \cdot x^{-n+1} \sum_{j=0}^\infty P[j,-n+1,\bs] t^j x^{-j} \\
			& \sim \frac14 \sum_{m=0}^\infty \sum_{n=0}^m \sum_{k=0}^{m-n} a_n^K
				P[k,n,\bd] P[m-n-k,-n+1,\bs] t^m x^{-m+1}\\
			&\sim \sum_{m=0}^\infty a_m^R t^m x^{-m+1}.
	\end{align*}
	
\end{proof}

With assumptions from Theorem \ref{thm-asymexp-resultant-all},
the first few coefficients in the expansion \eqref{exp-resultant-allpowers} of the resultant mean are given by:
 	\be\label{coeff-aR-all}
 	\begin{aligned}
 		a_0^R &= 1,\\
 		a_1^R &= \tfrac12 \bigl( a^K_1+a^M_1+a^N_1-a^K_1 a^M_1 a^N_1 \bigr),\\
 		a_2^R &= \tfrac14 \bigl( a^N_2  (2-2 a^K_1  a^M_1 )+a^K_2  (-1+a^M_1  a^N_1 )^2+a^M_2 + a^M_2 a^N_1 (a^N_1 -2 a^K_1)   \bigr),\\
 		a_3^R &= \tfrac1{64} \bigl( -32  a^N_3
 			( a^K_1   a^M_1 -1 ) -8  a^K_3  ( a^M_1 a^N_1 -1)^3 \\
 			&\qquad - 8  a^K_2  (-1+ a^M_1   a^N_1 ) ( (a^N_1)^2  a^M_1 -a^M_1(4a^N_2+1)   + a^N_1  \bigl((a^M_1)^2-4  a^M_2 -1)\bigr)\\
 			&\qquad +8 a^M_2 ( a^K_1 - a^N_1 ) \bigl((a^N_1)^2-4  a^N_2 -1\bigr)
 			-8a^M_3 (a^K_1 a^N_1(3+ (a^N_1)^2)-3  (a^N_1)^2 -1)    \bigr).
 	\end{aligned}
 	\ee
 	
\begin{rem}\label{rem-Cases}
	In the previous Theorem formula for the Case I: $a_1^N \neq \pm1$ was presented. In other cases, first few terms in the sequence $\bg$ or $\tilde{\bg}$ is equal to 0 and therefore the expression $P[\cdot,\cdot,\bg]$ or $P[\cdot,\cdot,\tilde{\bg}]$ is not well defined. In order to complement the result of Theorem \ref{thm-asymexp-resultant-all}, let $z$ be the minimal integer, greater or equal to 2, such that $a^N_z\neq 0$. In Case II: $a^N_1=-1$, let us redefine $\bg=(a^N_z,a^N_{z+1},\ldots)$ and let $\bh,\tilde{\bg},\tilde{\bh}$ be the same as in \eqref{thm-def-bgbh}. Then
	\begin{align*}
		d_{m-1}&= \sum_{n=0}^m a_n^M \sum_{k=0}^{m-n}
			 	P[k,n,\tilde{\bg}] P[m-n-k,-n+1,\tilde{\bh}] \\
			 	&\qquad\qquad
			 		-\sum_{n=0}^{\lfloor \frac{m}z \rfloor} a_n^M \sum_{k=0}^{m-nz}
			 		P[k,n,\bg] P[m-nz-k,-n+1,\bh], \quad m\in\bN, \\
		s_m&= \sum_{n=0}^m a_n^M \sum_{k=0}^{m-n}
			 	P[k,n,\tilde{\bg}] P[m-n-k,-n+1,\tilde{\bh}] \\
			 	&\qquad\qquad
			 		+\sum_{n=0}^{\lfloor \frac{m}z \rfloor} a_n^M \sum_{k=0}^{m-nz}
			 		P[k,n,\bg] P[m-nz-k,-n+1,\bh], \quad m\in\bN_0,
	\end{align*}
	and the coefficients of the resultant mean-map can be calculated by the formula \eqref{thm-asymexp-resultant-all-amR}.
	In Case III: $a^N_1=1$, the sequence $\tilde{\bg}$ needs to be redefined in a similar manner and the rest of the procedure goes analogously.
\end{rem}

\begin{rem}
	Once calculated by the procedure given in Theorem \ref{thm-asymexp-resultant-all}, the coefficients behave well for the special values mentioned in Cases II and III from Remark \ref{rem-Cases} where the existence has been shown.
	As it was proved in \cite[Lemma 2.2]{ElVu-2014-04} we may use coefficients from the list  \eqref{coeff-aR-all} with $a_1^N = \pm1$ and $a_2^N=\ldots=a_{z-1}^N=0$ to obtain the coefficients in Cases II and III as well.
\end{rem}

\begin{rem}\label{rem-stable}
Regarding the stability examination, when equating the coefficients $a^R_m$ from \eqref{coeff-aR-all} with $a^M_m=a^N_m=a^K_m$, for $m=1$ we obtain that $a_1^M(1+a_1^M)(1-a_1^M)=0$. The first possibility ($a_1^M=0$) implies (inductively) that also $a_{2m+1}=0$, for all $m\in\bN$, which means that $M$ has the asymptotic expansion of the type \eqref{def-MNK-M}. Each of other two possibilities ($a_1^M = \pm1$) implies that $a_{m}=0$, for all $m\ge2$, which correspond to the first and the second projection.
\end{rem}

	As a consequence of the coefficients comparison, using the stable mean expansion \eqref{exp-Mtt-stable}, for means whose asymptotic expansion does not contain even powers of $x$, on the one side and the corresponding expansions given in Proposition \ref{prop-exp-Mi} on the other side, and employing Remark \ref{rem-stable} for those means whose asymptotic expansions contain all powers $x^{-n+1}$, $n\in\bN_0$, we have the following conclusion.
\begin{cor}
None of the means form the list \eqref{mean-def-Mi} is stable.
\end{cor}

We finish this Section with a result based on Theorem \ref{exp-resultant-allpowers} which
 will be used in sequel.
 With $K=B_p$ and $N=B_q$, by incorporating coefficients \eqref{asymexp-powermean} into the expansions \eqref{def-MNK-Kt} and \eqref{def-MNK-Nt} with parameters $p$ and $q$ respectively, we obtain
coefficients in the asymptotic expansion \eqref{exp-resultant-allpowers} of the resultant mean--map:
 	\be\label{coeff-RBpNBq-all}
 	\begin{aligned}
 	 \cR(B_p,M,B_q) &(x-t,x+t)  = x+\tfrac12 a^M_1 t
 	 	+\tfrac18 \bigl(2a^M_2 +p+2q-3 \bigr) t^2 x^{-1}\\
 	 	& \qquad \qquad +\tfrac1{16} \bigl(2 a^M_3 -(p-1)(2q-1) a^M_1\bigr) t^3 x^{-2}
 	 	+\O(x^{-3}).
 	\end{aligned}
 	\ee

\section{Sub-stabilizability and super-stabilizability with power means }
In this Section we present some of the possible applications of the previously obtained results. We will show the use of asymptotic expansions when examining the possibility for a given mean to be $(B_p,B_q)$-sub- or super-stabilizable and when possible, how to find optimal parameters $p$ and $q$.

Since the asymptotic inequality is the necessary condition for the proper inequality, we analyse when
\be\label{sub-super-stab-power-necessary}
	M-\cR(B_p,M,B_q) \succ0.
\ee
The best approximation is obtained when as many as possible first coefficients are equal to 0. See \cite{Vu} for detailed analysis. Because of the reasoning in Remark \ref{rem-tt-st}, it is sufficient to use the expansions in variables $(x-t,x+t)$. 

Asymptotic inequality corresponds the proper inequality for means when variables $s$ and $t$ are close enough to each other. In order to complement the information obtained form asymptotic side, it is often useful to observe relation between means about the point $(0,1)$. For more details we refer to paper \cite{El-2015}.


\begin{exa}\label{exa-Lalpha}
	Let $M=L_\alpha$. Observe the difference between \eqref{exp-Mu-cosh} and \eqref{coeff-RBpNBq}.
When equating the coefficient by $x^{-1}$ with 0, we have that $q=\frac12(1-p)-2\alpha^2$ and then
	\bes
		 (L_\alpha-\cR(B_p,L_\alpha,B_q))(x-t,x+t) \sim
		 	-\tfrac1{384}(4\alpha^2-1)(p^2-1-16\alpha^2+16\alpha^4)t^4 x^{-3}+ \dots
	\ees
	and if also $p=\pm\sqrt{1+16\alpha^2-16\alpha^4}$ then
	\bes
		(L_\alpha-\cR(B_p,L_\alpha,B_q))(x-t,x+t) \sim 
		-\tfrac1{720}\alpha^2(\alpha^2-1)(4\alpha^2-1)^3 t^6 x^{-5} +\dots
	\ees
Hence, for such $p$ and $q$, \eqref{sub-super-stab-power-necessary} holds for $\frac12<\abs{\alpha}<1$, and we have the opposite inequality sign in \eqref{sub-super-stab-power-necessary} for $0<\abs{\alpha} <\frac12$.

Additionally, when we equate coefficient by $x^{-5}$ with 0, we obtain the following trivial (and already known) cases: $\alpha=\pm\frac12,q=-p$; $\alpha=0,p=-1,q=1$; $\alpha=0,p=1,q=0$; $\alpha=\pm1, p=1,q=-2$; where the stabilizability is achieved.

On the other side, values of limits $\lim_{s\to0} L_\alpha (s,1-s)=0$ and $\lim_{s\to0} \cR(B_p,L_\alpha,B_q) (s,1-s) >0$ for any $p$ and $q$, imply that $L_\alpha$ could only be $(B_p,B_q)$-super-stabilizable.  

Combining the observations about the sign of the difference $L_\alpha-\cR(B_p,L_\alpha,B_q)$ in two limiting points, with the intent to achieve the best possible order of inequality, 
we see that for $\frac12<\abs{\alpha}<1$ and above mentioned specific values of $p$ and $q$ mean $L_\alpha$ cannot be either sub- or super-stabilizable with the pair of power means $(B_p,B_q)$. 
\end{exa}

Motivated by the Example \ref{exa-Lalpha}, numerical experiments indicate that the following statement should be true.

\begin{conj}
	Let $q=\frac12(1-p)-2\alpha^2$ and $p=\pm\sqrt{1+16\alpha^2-16\alpha^4}$.
	Then $L_\alpha-\cR(B_p,L_\alpha,B_q)<\,0$ for $\lvert\alpha\rvert <\,\frac12$.
\end{conj}


\begin{exa}
Let $M=S_\alpha$. Observe the difference between \eqref{exp-Mu-1cosh} and \eqref{coeff-RBpNBq}. 
The best order of approximation as $x\to\infty$ is obtained when $q=\frac12(1-p)+2\alpha^2$
and $p^2=\frac{1-12\alpha^2+112\alpha^4-64\alpha^6}{1+4\alpha^2}$. For such values $p$ and $q$ we have
	\bes
		(S_\alpha-\cR(B_p,S_\alpha,B_q))(x-t,x+t) \sim
			\tfrac1{720} \alpha^2(1+\alpha^2)(1-16\alpha^2+16\alpha^4)^2(1+4\alpha^2)^{-1} t^6 x^{-5}+\dots
	\ees
Then, \eqref{sub-super-stab-power-necessary} holds when $\abs\alpha \neq \frac12\sqrt{2\pm\sqrt3}$ and $\alpha\neq0$. Examining the sign of the coefficient by $x^{-7}$ for the corresponding values of $\alpha$ we may see that the opposite sign in \eqref{sub-super-stab-power-necessary} holds for $\abs\alpha=\frac12\sqrt{2\pm\sqrt3}$. When additionally equating the coefficient by $x^{-5}$ with 0 we obtain trivial cases: $\alpha=0,p=1,q=0$ and $\alpha=0,p=-1,q=1$, where the stabilizability is achieved.
\end{exa}

\begin{conj}
	Let $q=\frac12(1-p)+2\alpha^2$ and $p^2=\frac{1-12\alpha^2+112\alpha^4-64\alpha^6}{1+4\alpha^2}$.
	Then $S_\alpha-\cR(B_p,S_\alpha,B_q)<0$ for $\abs{\alpha}=\frac12\sqrt{2\pm\sqrt3}$.
\end{conj}


\begin{exa}\label{exa-Sec4-M1}
Let $M=M_1$. Mean $M$ has the asymptotic expansion of the form \eqref{def-MNK-Mt}
and hence from \eqref{coeff-RBpNBq-all} follows that
\bes
	(M_1-\cR(B_p,M_1,B_q))(x-t,x+t)
		\sim  \tfrac12 a^M_1 t +\O(x^{-1}).
\ees
From Proposition \ref{prop-exp-Mi} we see that $a^M_1 = \frac{\lvert t \rvert}t$ 
so the expansion of the difference starts with coefficient $\tfrac12 a^M_1 t = \tfrac12 \abs{t}$ by $x^0$,
and the asymptotic inequality \eqref{sub-super-stab-power-necessary} holds,
which means that for $s$ and $t$ close enough and $s\neq t$, the difference
$M_1(s,t)-\cR(B_p,M_1,B_q)(s,t)$ is greater than 0. On the other side, observing the
	values of limits $\lim_{s\to0}M_1(s,1-s)=0$ and $\lim_{s\to0}\cR(B_p,M_1,B_q) (s,1-s)= 2^{-1/p}\cdot\lvert 2^{-1/q}-1 \rvert / \ln(1+\lvert\ln 2^{-1/q} \rvert ) >0$, $\forall p,q \in \bR$, implies that there are $s$ and $t$ for which the difference between $M_1$ and $\cR(B_p,M_1,B_q)$ is negative. We may conclude that $M_1$ cannot be either sub- or super-stabilizable with power means.
\end{exa}

\begin{exa}
Let $M=M_2$. Then from Proposition \ref{prop-exp-Mi} and \eqref{coeff-RBpNBq} we have
\begin{align*}
	(M_2-&\cR(B_p,M_2,B_q))(x-t,x+t)
		\sim  \tfrac18 (5-p-2q)t^2x^{-1} \\
		& + \tfrac{1}{384} \left(2 p^3-3 p^2+2 p (14 q-9)+4 q (4 (q-2) q-9)-45\right)t^4x^{-3}+\dots
\end{align*}
For $p=5-2q$ and $q=\frac12(5\pm\sqrt{17})$, coefficients by $x^{-1}$ and $x^{-3}$ become equal to 0, and coefficient by $x^{-5}$ is equal to $-\frac{11}{180} t^6$. Therefore,
\bes
	M_2 -\cR(B_p,M_2,B_q) \prec 0.
\ees
\end{exa}

\begin{conj}
	Let $p=5-2q$ and $q=\frac12(5\pm\sqrt{17})$.
	Then $M_2-\cR(B_p,M_2,B_q)<0$.
\end{conj}

\begin{exa}
Let $M=M_3$. Then we have the similar situation as in the Example \ref{exa-Sec4-M1} with the same value of the coefficient $a^M_1$, and therefore asymptotic inequality \eqref{sub-super-stab-power-necessary} holds.
On the other side, as $(s,t)\to(0,1)$, values of the difference $M_3-\cR(B_p,M_3,B_q)$ may be positive or negative, depending on $p$ and $q$. Hence, $M_3$ could only be sub-stabilizable with power means.
\end{exa}

\begin{exa}
Let $M=M_4$. Then
\begin{align*}
	(M_4&-\cR(B_p,M_4,B_q))(x-t,x+t)
		\sim  \tfrac18 (3-p-2 q)t^2x^{-1} \\
		& + \tfrac{1}{384} \left(p (2+p) (-7+2 p)+12 p q-24 q^2+16 q^3-7 (3+4 q)\right)t^4x^{-3}+\dots
\end{align*}
For $p=3-2q$ and $q=\frac12(3\pm\sqrt{21})$ coefficients by $x^{-1}$ and $x^{-3}$ become equal to 0, the coefficient by $x^{-5}$ is equal to $\frac{13}{120} t^6$ and 
therefore the asymptotic inequality \eqref{sub-super-stab-power-necessary} holds.
\end{exa}

\begin{conj}
	Let $p=3-2q$ and $q=\frac12(3\pm\sqrt{21})$.
	Then $M_4-\cR(B_p,M_4,B_q)>0$.
\end{conj}

\begin{exa}
Mean $M=M_5$, when compared with the resultant mean--map behaves similarly as mean $M_1$. Namely, asymptotic inequality \eqref{sub-super-stab-power-necessary} holds,
and further,
	$\lim_{s\to0}M_5(s,1-s)=0$ and $\lim_{s\to0}\cR(B_p,M_5,B_q) (s,1-s) >0$, $\forall p,q \in \bR$.
	Analogously as in the Example \ref{exa-Sec4-M1}, we may conclude that mean $M_5$ cannot be either $(B_p,B_q)$-sub or super-stabilizable.
\end{exa}

\begin{exa}
Let $M=M_{\alpha,r}$, defined in \eqref{mean-def-Mi} with special and limit cases described in Proposition \ref{prop-Malphar-specialcases}. From \eqref{coeff-RBpNBq-all} we find the asymptotic expansion of the difference
	\bes
		(M_{\alpha,r}-\cR(B_p,M_{\alpha,r},B_q))(x-t,x+t)
			\sim \tfrac12 a_1^M t  
			+\dots,
	\ees
	where from Proposition \ref{prop-exp-Mi} we see that $a_1^M=\alpha \frac{\lvert t\rvert}t$. Therefore,
	asymptotic inequality \eqref{sub-super-stab-power-necessary} holds when $\alpha > 0$, and the opposite asymptotic inequality holds when $\alpha < 0$.

	Observing the limit of $M_{\alpha,r}(s,1-s)$ as $s\to0$, we see that it is equal to $0$ for $r+\alpha\ge0$ and $-(r+\alpha)$ for $r+\alpha<0$.
	
	If $\alpha>0$, then by the same argument as in the Example \ref{exa-Sec4-M1},
	mean $M_{\alpha,r}$ cannot be either $(B_p,B_q)$-sub- or super-stabilizable.
	If $\alpha<0$, then mean $M_{\alpha,r}$ can only be $(B_p,B_q)$-super-stabilizable.
	If $\alpha=0$, then $M_{\alpha,r}$ corresponds the logarithmic mean $L$, whose asymptotic expansion can be found in \cite{ElVu-2014-09}:
	\bes
		L(x-t,x+t)\sim x -\tfrac13 t^2x^{-1} -\tfrac4{45} t^4 x^{-3}
			- \tfrac{44}{945} t^6 x^{-5} +\O(x^{-7}).
	\ees
	Now we have
	\begin{align*}
		(L&-\cR(B_p,L,B_q))(x-t,x+t)
			\sim \tfrac18(1-p-2q) t^2x^{-1} \\
			&\qquad +\tfrac1{384}\big(2p^3 -3 p^2 -2p (5+2 q)+4q(4q(q-1)-5) +11 \big) t^4x^{-3}
			+\O(x^{-5}).
	\end{align*}
	If we set $p=1-2q$, then
		\bes
		(L-\cR(B_{1-2q},L,B_q))(x-t,x+t)
			\sim \tfrac1{96} q(q-1) t^4x^{-3}
			+\O(x^{-5}),
		\ees
	and if, additionally, $q=0$ or $q=1$, we obtain two pairs of power means for which the resultant mean $\cR(B_p,L,B_q)$ is equal to $L$:
		\bes
			\cR(A,L,G)=L,\quad \cR(H,L,A)=L,
		\ees
		which can be easily verified by direct computation.
\end{exa}





\end{document}